\documentclass[reqno]{amsart}
\usepackage{hyperref}

\begin{document}
\title[Some spectral properties and isoperimetric inequalities]
{Some spectral properties and isoperimetric inequalities for a nonlocal Laplacian problem}

\author[M.\,A.~Sadybekov and
B.\,T.~Torebek\hfil \hfilneg] {Makhmud~A.~Sadybekov and
Berikbol~T.~Torebek}  

\address{Makhmud~A.~Sadybekov \newline
Institute of Mathematics and Mathematical Modeling.\newline 125
Pushkin str., 050010 Almaty, Kazakhstan} \email{sadybekov@math.kz}

\address{Berikbol T. Torebek \newline
Institute of Mathematics and Mathematical Modeling.\newline 125
Pushkin str., 050010 Almaty, Kazakhstan} \email{torebek@math.kz}


\begin{abstract}
In this paper we consider a non-local problem for a Laplace
operator  in a multidimensional bounded symmetric domain. The
investigated problem is an analogue of the classical periodic
boundary value problems in the case of non-rectangular domain. We
prove self-adjointness of the problem and show a method of
constructing eigenfunctions. We obtain an analogue of the Rayleigh
type inequality and some spectral inequalities for the first
eigenvalue of the nonlocal problem.\end{abstract}

\maketitle \numberwithin{equation}{section}
\newtheorem{theorem}{Theorem}[section]
\newtheorem{corollary}[theorem]{Corollary}
\newtheorem{lemma}[theorem]{Lemma}
\newtheorem{remark}[theorem]{Remark}
\newtheorem{problem}{Problem}
\newtheorem{example}[theorem]{Example}
\allowdisplaybreaks

\section{Introduction}\label{sec1}
Eigenvalues of the Laplacian represent frequencies in wave motion,
rates of decay in diffusion, and energy levels in quantum
mechanics. Constructing eigenvalues is a difficult task: they are
known in an explicit form only for some special domains. This fact
leads to necessity of considerable work on estimating eigenvalues
in terms of simpler geometric quantities such as area and
perimeter.

Let $\Omega  \subset \mathbb{R}^n $ be a bounded domain, symmetric
with respect to the origin and with a smooth boundary $\partial
\Omega .$ This symmetry means that alongside with a point
$(x_1,...,x_n )$ her  "opposite" point $x^* = ( - x_1 ,..., - x_n
)$ also belongs to the domain. Let us denote $\partial \Omega _ +
= \partial \Omega  \cap \{ x_1  \ge 0\} ,$ $\partial \Omega _ -
= \partial \Omega  \cap \{ x_1  < 0\} .$

Consider the following problem:\\
{\bf Problem P} {\it Find a function  satisfying the equation
\begin{equation}\label{1}- \Delta u\left( x \right) = f\left( x \right),x \in \Omega,\end{equation}
in the domain $\Omega,$ and satisfying the following boundary conditions
\begin{equation}\label{2}u\left( x \right) =  - u\left( {x^*} \right),x \in \partial \Omega _ +,\end{equation}
\begin{equation}\label{3}\frac{{\partial u\left( x \right)}}{{\partial n_x }} =
\frac{{\partial u\left( {x^*} \right)}}{{\partial n_x }},x \in \partial \Omega _ +  .\end{equation}}
Here $n_x$ is a derivative in the direction of an outer normal to $\partial \Omega.$

Note that the problem P in the case of a disk, and a
multidimensional ball was first formulated and investigated in
\cite{SadybekovTurmetov:2012}-\cite{SadybekovTurmetovTorebek:2014}.

The objectives of this work are:
\begin{itemize}
\item proving self-adjointness of the problem P;
\item obtaining analogue of the Rayleigh type inequality for the first eigenvalue of problem P and of the inverse operator to problem P;
\item constructing all the eigenfunctions of the problem P;
\item proving some spectral inequalities for the first and second eigenvalues of the problem P;
\item getting estimates for the norm of the inverse operator to problem P.
\end{itemize}

Note that the isoperimetric and asymptotic inequalities for the
eigenvalues of the Laplace operator and of the convolution type
operators were obtained in
\cite{GesztesyMitrea:2009}-\cite{Henrote:2006}.

Using the method of \cite{SadybekovTurmetov:2014}, we can be prove
\begin{theorem}\label{th1.1} For any $f \in L_2 \left( \Omega  \right)$
a strong solution of problem P exists and is unique. This solution belongs to the class $W_2^2 (\Omega )$ and it is represented as \begin{equation}\label{4}u(x) =\mathcal{L}_\Omega^{-1}f(x)= \int\limits_\Omega  {G_P (x,y)f(y)dy},\end{equation}
where $G_P (x,y)$ is a Green's function of the problem P, which has the form:
\begin{equation}\label{5}G_P (x,y) = \frac{1}{2}\left[ {G_D (x,y) + G_D (x,y^*) + G_N (x,y) - G_N (x,y^*)} \right].\end{equation}
Here $G_D (x,y)$ is the Green's function of the Dirichlet problem
and $G_N (x,y)$ is the Green's function of the Neumann problem.
\end{theorem}

It is well known that the Green's function of the Dirichlet
problem in a unit ball of arbitrary dimension can be constructed
by reflection method. Note that the Green's function of the
Neumann problem in the unit ball of arbitrary dimension can be
constructed explicitly \cite{SadybekovTurmetovTorebek:2016}. Thus,
the Green's function of the problem P in the unit ball of
arbitrary dimension can be constructed explicitly.

\section{Main results}
Let $\lambda _1^D \left( \Omega  \right)$ be a first eigenvalue of
the Dirichlet problem, and $\lambda _2^N \left( \Omega  \right)$
be a second eigenvalue of the Neumann problem  in $\Omega.$ Also
we denote by $\mu_1\left( \Omega  \right)$ a first eigenvalue and
by $\mu_2\left( \Omega  \right)$ a second eigenvalue of the
integral operator $\mathcal{L}_\Omega^{-1}.$
\begin{theorem}\label{th4.1} Let $\lambda _1 \left( \Omega  \right)$ be a   first eigenvalue of the problem P. Then $$\lambda _1 \left( \Omega  \right) = \lambda _2^N \left( \Omega  \right),$$ $$\lambda _1 \left( \Omega  \right) < \lambda _1^D \left( \Omega  \right).$$\end{theorem}

\begin{theorem}\label{th4.2} Let $B$ be a ball of the same measure as $\Omega,$ i.e $\left| B \right| = \left| \Omega  \right|.$
Then for the first eigenvalue of the problem P the Rayleigh type inequality:
\begin{equation}\label{13}\lambda _1 \left( \Omega  \right) \le \lambda _1 \left( B \right)\end{equation} is true.\\
That is, the ball maximizes the first eigenvalue of the problem P among all domains with equal measure.\end{theorem}

By Theorem \ref{th4.2} the following results take place:
\begin{corollary}\label{cor4.1}Let $B$ be a ball of the same measure as $\Omega,$
i.e $\left| B \right| = \left| \Omega  \right|.$ Then for the first eigenvalue of the operator
$\mathcal{L}_\Omega^{-1}$ the following inequality:
\begin{equation*}\mu_1 \left( \Omega  \right) \geq \mu _1 \left( B \right)\end{equation*} is true.\\
That is, the ball minimizes the first eigenvalue of the operator
$\mathcal{L}_\Omega^{-1}$ among all domains with equal measure.
\end{corollary}

\begin{corollary}\label{cor4.2}Let $D$ be a disk of the same measure
as $\Omega\in \mathbb{R}^2,$ i.e $\left| D \right| = \left| \Omega
\right|.$ Then for the first and second eigenvalues of the
operator $\mathcal{L}_\Omega^{-1}$ the following inequality:
\begin{equation*}\mu_1 \left( \Omega  \right) + \mu_2 \left(
\Omega  \right) \geq \mu _1 \left( D \right)+\mu _2 \left( D
\right),\end{equation*} holds.
\end{corollary}

\begin{corollary}\label{cor4.3}Let $D$ be a disk with given diameter $d$
of the same measure as $\Omega\in \mathbb{R}^2,$ i.e $\left| D
\right| = \left| \Omega  \right|.$ Then the norm of the operator
$\mathcal{L}_\Omega^{-1}$ is given by:
\begin{equation*}\|\mathcal{L}_\Omega^{-1}\|\geq \frac{\pi
p}{d},\end{equation*} where $p=1.8412...$ is the first positive
zero of the derivative of the Bessel function.
\end{corollary}

Results of Corollary \ref{cor4.2} and Corollary \ref{cor4.3}
follow from Theorem Szeg\"{o} \cite{Henrote:2006}.

\begin{theorem}\label{th4.3} Let $B$ be a ball of the same measure as $\Omega,$
i.e $\left| B \right| = \left| \Omega  \right|.$ Then the
following inequalities hold:
\begin{equation}\label{14}\frac{{\lambda _1 \left( \Omega  \right)}}{{\lambda _1^D \left( \Omega  \right)}} \le \frac{{\lambda _1 \left( B \right)}}{{\lambda _1^D \left( B \right)}},\end{equation}
\begin{equation}\label{15}\lambda _1^D \left( \Omega  \right) -
\lambda _1 \left( \Omega  \right) \ge \lambda _1^D \left( B
\right) - \lambda _1 \left( B \right).\end{equation}\end{theorem}

The proofs of theorems are given in section 6. First of all, we
consider the method of constructing eigenfunctions of the problem
P.

\section{Operator properties of the problem P}
\begin{theorem}\label{th2.1} The problem P is self-adjoint in Hilbert space $L_2 (\Omega )$.\end{theorem}
\begin{proof} By  $\mathcal{L}_\Omega$ we denote a closure in $L_2 (\Omega )$ of the operator
defined by the differential expression $$\mathcal{L}_\Omega u =  - \Delta u(x)$$ on the linear manifold of
functions $u(x) \in C^{\alpha  + 2} \left( {\bar \Omega } \right),\,0 < \alpha  < 1$ satisfying the
boundary conditions \eqref{2} and \eqref{3}. By Theorem \ref{th1.1}, this operator is invertible
and $D(\mathcal{L}_\Omega) \subset W_2^2 (\Omega ).$ Therefore $\mathcal{L}_\Omega^{ - 1} $ is a
compact operator in $L_2 (\Omega ).$

In addition, the inverse operator admits the integral
representation \eqref{4}. The self-adjointness of the operator
$\mathcal{L}_\Omega^{ - 1} $ is easily obtained from \eqref{5}.
This follows from the symmetry of the kernel $G_P (x,y)$ of the
integral operator \eqref{4}. Consequently, the operator
$\mathcal{L}_\Omega$ is self-adjoint as well. The proof of theorem
\ref{th2.1} is complete.\end{proof}

\begin{theorem}\label{th2.2} The operator $\mathcal{L}_\Omega,$ corresponding to problem P,
is positively definite, that is $$\left( {\mathcal{L}_\Omega u,u}
\right) \ge \left\| u \right\|^2$$ for all $u\left( x \right) \in
D\left( {\mathcal{L}_\Omega} \right).$\end{theorem}
\begin{proof} It is easy to show that for any  $u \in D\left( {\mathcal{L}_\Omega} \right)$
the following inequality $$\left( {\mathcal{L}_\Omega u,u}
\right)_{L_2 \left( \Omega  \right)}  = \int\limits_\Omega
{|\nabla u(y)|^2 dy}  \ge 0$$ holds. This fact, by virtue of the
Friedrichs type inequalities, implies the assertion of theorem
\ref{th2.2}.\end{proof}

\section{Construction of eigenfunctions of the problem P.}

Since $\mathcal{L}_\Omega^{ - 1} $ is a compact operator, we find that its spectrum (and, consequently, the spectrum of the operator $\mathcal{L}_\Omega$) can consist only of real eigenvalues of finite multiplicity. This operator has no associated functions. By Theorem \ref{th2.2} all the eigenvalues of the operator $\mathcal{L}_\Omega$  are positive. Consequently, they can be numbered in order of increasing $$0 < \lambda _1  < \lambda _2  < ... < \lambda _k  < ...,k \in \mathbb{N}.$$ The eigenfunctions of the operator $\mathcal{L}_\Omega$ are found as solutions of the equation \begin{equation}\label{6} - \Delta u\left( x \right) = \lambda u\left( x \right),x \in \Omega ,\end{equation} satisfying the boundary conditions \eqref{2} and \eqref{3}.

For Eq. \eqref{6} consider two auxiliary problems:\\
Dirichlet problem
\begin{equation}\label{7} - \Delta v\left( x \right) = \lambda v\left( x \right),x \in \Omega ;v\left( x \right) = 0,x \in \partial \Omega ,\end{equation}
and Neumann problem
\begin{equation}\label{8} - \Delta w\left( x \right) = \lambda w\left( x \right),x \in \Omega ;\frac{{\partial w}}{{\partial n_x }}\left( x \right) = 0,x \in \partial \Omega .\end{equation}
Both of the problems are self-adjoint, hence root subspaces consist only of eigenfunctions. All eigenvalues of the problems (except the first) are multiples.

The Dirichlet and Neumann problems possess the following symmetry properties of eigenfunctions.
\begin{lemma}\label{lem3.1} All eigenfunctions of the Dirichlet problem \eqref{7}
and of the Neumann problem \eqref{8} can be chosen so that they
have one of the symmetry properties:
\begin{equation}\label{9}v(x) + v(x^*) = 0,\end{equation}
or \begin{equation}\label{10}v(x) - v(x^*) =
0.\end{equation}\end{lemma}
\begin{proof} We give the proof of lemma \ref{lem3.1} only for the Dirichlet problem.
Suppose that $\left\{ {v_k (x)} \right\}_{k \in \mathbb{N}} $
 is a system of all normalized eigenfunctions of the Dirichlet problem \eqref{7}, and $\lambda _k^D $
 are corresponding eigenvalues. Denote
$$v_{kj} \left( x \right) = v_k \left( x \right) + \left( { - 1} \right)^j v_k \left( {x^*} \right),j = 0,1.$$

It may be that $v_{k0} \left( x \right) \equiv 0$ or $v_{k1}
\left( x \right) \equiv 0,$ but not simultaneously. Obviously, all
these functions satisfy one of the symmetry conditions, that is,
either condition \eqref{9} (if $j = 1$) or condition \eqref{10}
(at $j = 0$). It is easy to see that $v_{kj} \left( x \right)$ are
solutions of the Dirichlet problem $$ - \Delta v_{kj} \left( x
\right) = \lambda _k^D v_{kj} \left( x \right),x \in \Omega;$$
$$v_{kj} \left( x \right) = 0, x \in
\partial \Omega ,j = 0,1.$$ Therefore, if they are different from
zero, then they are eigenfunctions.

Let us consider a system of functions
\begin{equation}\label{11}\left\{ {v_{k0} (x),v_{k1} (x)}
\right\}_{k \in \mathbb{N}}.\end{equation} A part of these
functions may be zero, but we not pay attention on this fact. We
show that system \eqref{11} is complete in $L_2 \left( \Omega
\right).$ Indeed, suppose $g\left( x \right) \in L_2 \left( \Omega
\right)$ is orthogonal to all functions of  system \eqref{11}.
Then for all $k \in \mathbb{N}$ we get $$0 = \left( {v_{kj} ,g}
\right) = \int\limits_\Omega  {v_{kj} \left( x \right)\overline
{g\left( x \right)} dx}  = \int\limits_\Omega  {\left[ {v_k \left(
x \right) + \left( { - 1} \right)^j v_k \left( {x^*} \right)}
\right]\overline {g\left( x \right)} dx} $$ $$ =
\int\limits_\Omega  {v_k \left( x \right)\left[ {\overline
{g\left( x \right)}  + \left( { - 1} \right)^j \overline {g\left(
{x^*} \right)} } \right]dx} ,j = 0,1.$$

Since the system of all eigenfunctions of the Dirichlet problem
$\left\{ {v_k (x)} \right\}_{k \in \mathbb{N}} $ is complete in
$L_2 \left( \Omega  \right),$ then $g\left( x \right) \equiv 0$
for almost all $x \in \Omega .$ This fact proves the completeness
of system \eqref{11} in $L_2 \left( \Omega  \right).$ System
\eqref{11} remains complete under removing zero functions from it.
All non-zero functions of system \eqref{11} are the eigenfunctions
of the Dirichlet problem \eqref{7}. Since system \eqref{11} is
complete in $L_2 \left( \Omega  \right),$ then the problem has no
other (linearly independent) eigenfunctions. All system components
\eqref{11} have the symmetry property \eqref{9} or \eqref{10}.
This proves Lemma \ref{lem3.1} for the case of the Dirichlet
problem.

For the Neumann problem \eqref{8} the proof is analogous. Lemma
\ref{lem3.1} is proved.\end{proof}

\begin{remark}
The term "chosen so" means that the problem can have its
eigenfunctions  not satisfying \eqref{9} or \eqref{10}. But in
every root subspace one can choose other eigenfunctions, which
will already satisfy \eqref{9} or \eqref{10}.
\end{remark}

\begin{theorem}\label{th3.1} Eigenvalues of the problem P form two series. Accordingly, the system of eigenfunctions also consists of two series. The first series is the eigenfunctions of the Dirichlet problem \eqref{7} having the symmetry property \eqref{10}. The second series is eigenfunctions of the Neumann problem \eqref{8} having the symmetry property \eqref{9}.\end{theorem}
\begin{proof} Let $\left\{ {v_{k0} (x)} \right\}_{k \in \mathbb{N}} $ be a
system of eigenfunctions of the Dirichlet problem \eqref{7} having the
symmetry property \eqref{10}, and let $\lambda _{k0}^D$ be corresponding eigenvalues of the Dirichlet problem.
Suppose $\left\{ {w_{k1} (x)} \right\}_{k \in \mathbb{N}} $ is a system of eigenfunctions of the Neumann problem \eqref{8}
having the symmetry property \eqref{9}, $\lambda _{k1}^N $ are corresponding eigenvalues of the Neumann problem.
Obviously, the functions $v_{k0} (x)$ and $w_{k1} (x)$ are eigenfunctions of the non-local problem
\eqref{6}, \eqref{2}, \eqref{3} corresponding eigenvalues $\lambda _{k0}^D $ and $\lambda _{k1}^N.$
We show that the system
\begin{equation}\label{12}\left\{ {v_{k0} \left( x \right),w_{k1} \left( x \right)} \right\}_{k \in \mathbb{N}}\end{equation}
is complete in $L_2 \left( \Omega  \right).$ Then the problem P
does not have other eigenfunctions.

We divide the space $L_2 \left( \Omega \right)$  into a direct sum
of two spaces: spaces $L_2^ + \left( \Omega  \right)$ of functions
with the symmetry property \eqref{10} and spaces $L_2^ -  \left(
\Omega  \right)$ of functions with the symmetry property
\eqref{9}. By Lemma \ref{lem3.1}, the system $\left\{ {v_{k0} (x)}
\right\}_{k \in \mathbb{N}} $ is complete in $L_2^ +  \left(
\Omega  \right),$ and the system $\left\{ {w_{k1} (x)} \right\}_{k
\in \mathbb{N}} $ is complete in $L_2^ -  \left( \Omega  \right).$
Consequently, system \eqref{12} is complete in $L_2 \left( \Omega
\right).$ Therefore, the nonlocal problem \eqref{6}, \eqref{2},
\eqref{3} has no eigenfunctions of another type. The proof of
Theorem \ref{th3.1} is complete.\end{proof}

\section{Some spectral inequalities for the eigenvalues of the problem P}

\begin{proof}(Theorem \ref{th4.1}) From the results of Theorem \ref{th3.1}
it follows that problem \eqref{1}-\eqref{3}  has two series of
eigenvalues. The first series coincides with the eigenvalues of
the Neumann problem $\lambda _{k1}^N \left( \Omega  \right),$ and
eigenfunctions are selected from the eigenfunctions of the Neumann
problem with the symmetry property \eqref{9}. The second series
coincides with eigenvalues of the Dirichlet problem $\lambda
_{k0}^D \left( \Omega  \right),$ and eigenfunctions are selected
from the eigenfunctions of the Dirichlet problem with the symmetry
property \eqref{10}. Therefore it is easy to see that $\lambda
_1^{} \left( \Omega  \right) = \lambda _2^N \left( \Omega
\right).$

Also it is known \cite{Filonov:2005} that if $\Omega  \subset
\mathbb{R}^n $ is such that the embedding $W_2^1 \left( \Omega
\right) \subset L_2 \left( \Omega  \right)$ is compact, then for
all $k \in \mathbb{N}$ the inequality $$\lambda _{k + 1}^N \left(
\Omega  \right) < \lambda _k^D \left( \Omega  \right)$$ is true.
Consequently, $\lambda _1^{} \left( \Omega  \right) < \lambda _1^D
\left( \Omega  \right).$ This proves the theorem
\ref{th4.1}.\end{proof}

\begin{proof}(Theorem \ref{th4.2}) From the results of Theorem \ref{th4.1},
we have that the first eigenvalue of non-local problem P coincides
with the second eigenvalue of the Neumann problem: $\lambda _1^{}
\left( \Omega  \right) = \lambda _2^N \left( \Omega  \right).$ As
follows from \cite{Henrote:2006} (Szeg\"{o}-Weinberger Theorem),
the ball maximizes the second eigenvalue of the Neumann problem
among all domains with equal measure. Inequality \eqref{13} is
proved.\end{proof}

\begin{proof}(Theorem \ref{th4.3}) Consider the expression:
$\frac{{\lambda _1 \left( \Omega  \right)}}{{\lambda _1^D \left(
\Omega  \right)}}.$ It is known \cite{Henrote:2006} (Faber-Krahn
Theorem) that the ball minimizes the first eigenvalues of the
Dirichlet problem among all domains with equal measure. And, by
Theorem \ref{th4.2} we have inequality \eqref{13}. Then expression
\eqref{14} is obtained from the following chain of inequalities:
$$\frac{{\lambda _1 \left( \Omega  \right)}}{{\lambda _1^D \left( \Omega  \right)}}
\le \frac{{\lambda _1 \left( B \right)}}{{\lambda _1^D \left( \Omega  \right)}}
\le \frac{{\lambda _1 \left( B \right)}}{{\lambda _1^D \left( B \right)}}.$$
Inequality \eqref{15} is proved similarly.\end{proof}

\end{document}